\documentclass[reqno, 11pt, a4paper]{amsart} 

\usepackage[T5,T1]{fontenc}
\usepackage{mlmodern}

\usepackage[utf8]{inputenc}
\allowdisplaybreaks
\usepackage{amsfonts}
\usepackage{amsmath}
\usepackage{amssymb}
\usepackage{amsthm}
\usepackage[text={33pc,605pt},centering]{geometry}     

\usepackage{mathrsfs} 
\usepackage[dvipsnames]{xcolor}

\usepackage{bbm}
\usepackage{mathtools}

\usepackage{dsfont}
 
\usepackage[pagebackref=true,colorlinks=true, linkcolor=Blue, citecolor=Blue, pdfencoding=auto, psdextra]{hyperref}
\renewcommand*\backref[1]{\ifx#1\relax \else (Cited on #1) \fi}

\usepackage{appendix}
\usepackage{enumitem}
\usepackage{float}

\usepackage{tikz}
\usepackage{tikz-cd} 
\usetikzlibrary{arrows, arrows.meta}

\usepackage[nameinlink]{cleveref} 
\usepackage{scalerel}[2016/12/29]

\theoremstyle{plain}
\newtheorem{definition}{Definition}
\newtheorem{proposition}[definition]{Proposition}
\newtheorem{lemma}[definition]{Lemma}
\newtheorem{corollary}[definition]{Corollary}
\newtheorem{theorem}[definition]{Theorem}
\newtheorem{remark}[definition]{Remark}

\theoremstyle{definition}
\newtheorem{example}[definition]{Example}

\numberwithin{definition}{section}
\numberwithin{equation}{section}

\newcommand*{\R}{\mathbb{R}}

\newcommand*{\Z}{\mathbb{Z}}

\newcommand*{\Fcal}{\mathcal{F}}

\newcommand*{\Pcal}{\mathcal{P}}

\newcommand{\norm}[1]{\left \lVert  #1 \right \rVert}
\newcommand{\abs}[1]{\left\lvert #1 \right\rvert}

\newcommand*{\1}{\mathds{1}}

\renewcommand*{\d}{\mathrm{d}}
\newcommand*{\e}{\mathrm{e}}

\newcommand*{\SpecEnt}{\mathscr{I}}
\newcommand*{\RelEnt}{I}

\makeatletter
\newcommand{\subalign}[1]{%
  \vcenter{%
    \Let@ \restore@math@cr \default@tag
    \baselineskip\fontdimen10 \scriptfont\tw@
    \advance\baselineskip\fontdimen12 \scriptfont\tw@
    \lineskip\thr@@\fontdimen8 \scriptfont\thr@@
    \lineskiplimit\lineskip
    \ialign{\hfil$\m@th\scriptstyle##$&$\m@th\scriptstyle{}##$\hfil\crcr
      #1\crcr
    }%
  }%
}
\makeatother

\crefname{equation}{}{}

\title[Modified log-Sobolev, concentration, unique Gibbs measures]{Modified log-Sobolev inequalities, concentration bounds and uniqueness of Gibbs measures}

\author[Y. Steenbeck]{Yannic Steenbeck}

\address[Yannic Steenbeck]{TU Braunschweig, Institut für Mathematische Stochastik, 
Germany.}
\email{yannic.steenbeck@tu-braunschweig.de}


\date{\today}

\usepackage{soul}

\begin{document}

\begin{abstract}
    We prove that there is only one translation-invariant Gibbsian point process w.r.t.\ to a chosen interaction if any of them satisfies a certain bound related to concentration-of-measure.
    This concentration-of-measure bound is e.g.\ fulfilled if a corresponding modified logarithmic Sobolev inequality holds.\ In particular, for natural examples with non-uniqueness regimes, a modified logarithmic Sobolev inequality cannot be satisfied. Therefore, in these situations, the free-energy dissipation in related continuous-time birth-and-death dynamics in \(\mathbb{R}^d\) is not exponentially fast.
\end{abstract}

\maketitle


\section{Modified logarithmic Sobolev inequalities, related concentration bounds and uniqueness of Gibbs measures via specific relative entropy distance}\label[section]{section:introduction_and_main_theorems}

There is a well-known meta-correspondence between certain functional inequalities, the so-called concentration-of-measure phenomenon and trend to equilibrium for related dynamics. Each of these subjects is of major interest on its own, but their connections can be particularly exciting.
Perhaps the most prominent example of functional inequalities in this context are the {\em logarithmic Sobolev inequalities} discovered by Leonard Gross \cite{Gross1975} in their simplest form reading as follows.
Letting \(\Phi \colon \mathbb{R}_+ \to \mathbb{R}\) be given via \(\Phi(x) = x \log x\), \(x \in \mathbb{R}_+\), we define
\begin{align*}
    \mathrm{Ent}_{\nu}[F]
    = \nu\big[\Phi(F)\big] - \Phi\big(\nu[F] \big),
\end{align*} where \(\nu\) is the standard Gaussian measure on \(\R^d\) and \(F \colon \R^d \to \R\) a smooth and compactly supported function.
Furthermore, the (formal) Dirichlet form corresponding to the {\em Ornstein--Uhlenbeck semigroup}, admitting the standard Gaussian measure \(\nu\) as reversible measure, is given by
\begin{align*}
    \mathcal{E}^{\nu}(F, G)
    = \nu\big[\nabla F \cdot \nabla G \big]
    = \nu\big[F \,(-\mathscr{L} G)\big],
\end{align*} with the (formal) generator
\begin{align*}
    (\mathscr{L} F)(x)
    = \Delta F(x) - x \cdot \nabla F(x).
\end{align*}
Then, the following logarithmic Sobolev inequality 
\begin{align}\label{equation:logarithmic_sobolev_inequality_Gross}
    \mathrm{Ent}_\nu[F^2]
    \leq C \,\mathcal{E}^{\nu}(F, F)
\end{align} holds with \(C = 1\).
It turns out that this inequality generalizes to probability measures on \(\R^d\) other than the standard Gaussian one. 
Ira Herbst observed that a necessary condition for \Cref{equation:logarithmic_sobolev_inequality_Gross} to hold for a probability measure \(\nu\) on \(\R^d\) is that
\begin{align*}
    \nu\big[ \e^{\lambda (F - \mu[F])} \big]
    \leq \exp\big(\tfrac{C}{4} \lambda^2  \norm{F}_{\mathrm{Lip}}^2\big)
\end{align*} holds for every \(\lambda \in \R\) and Lipschitz-continuous \(F\). The connection with concentration inequalities via Chernoff bounds, i.e.\ the exponential Markov inequality, is well-understood.
The deep links between this family of functional inequalities and concentration-of-measure bounds for reversible measures as well as convergence of associated dynamics to equilibrium are for example discussed in \cite{Ledoux1999}. Since then, much research has occurred in this direction and the general correspondence can be transferred to many different situations. 

We will now recall a tiny bit of the corresponding explorations in the context of {\em Poisson point processes}. But first let us recall what a Poisson point process \(\pi\) of intensity one on \(\R^d\) is and embed it into a more general setup of notation. We will use the shorthand \(\Lambda \Subset \R^d\) to denote that \(\Lambda \subseteq \R^d\) is bounded and Borel-measurable. Let \(\Omega\) be the set of all simple, locally finite counting measures (point configurations) on \(\R^d\), equipped with the \(\sigma\)-algebra \(\mathcal{F}\) which is induced by all the counting variables \(\{N_\Lambda \,\colon\, \Lambda \Subset \R^d\}\), where \(N_\Lambda \colon \Omega \to [0, \infty), \omega \mapsto \#(\omega \cap \Lambda)\). We may identify \(\eta \in \Omega\) with its support and write \(\eta_{\Lambda} = \eta \cap \Lambda\) for the restriction of \(\eta\) to \(\Lambda \subseteq \R^d\). Similarly, for \(\eta, \zeta \in \Omega\) the point configuration \(\eta_{\Lambda} \zeta_{\Delta} \in \Omega\) is then given by the union of all the points of \(\eta\) in \(\Lambda\) and all the points of \(\zeta\) in \(\Delta\). It is common to say that a measurable \(F \colon \Omega \to \R\) is {\em local} iff there is a \(\Lambda \Subset \R^d\) such that \(F = F(\cdot_\Lambda)\). We will occasionally use the notations \(\nu\big[ F(\eta) \big] = \nu\big[F \big]\) for the expectation \(\int \, F(\eta) \, \nu(\d \eta)\) of a measurable \(F \colon \Omega \to \R\) w.r.t.\ a probability measure \(\nu\) on \((\Omega, \Fcal)\); the integration variable is then always called \(\eta\).
Now let \(\Pcal_\theta = \Pcal_\theta(\Omega, \Fcal)\) be the space of translation-invariant probability measures on \((\Omega, \mathcal{F})\), i.e.\ of those probability measures on \((\Omega, \mathcal{F})\) which are invariant under all translation maps \((\theta_x)_{x \in \R^d}\), acting as \(\theta_x \colon \Omega \to \Omega, \, \omega = \sum_{i} \delta_{x_i} \mapsto \theta_x \omega := \sum_{i} \delta_{x_i - x} \).
Now, the {\em Poisson point process of intensity one} on \(\R^d\), here denoted by \(\pi\), is the unique element \(\nu\) of \(\Pcal_\theta\) such that, under \(\nu\), \(N_\Lambda\) is Poisson-distributed with mean \(\abs{\Lambda}\) (where \(\abs{\cdot}\) means Lebesgue-measure on \(\R^d\)) and \(N_{\Lambda_1}, \dots, N_{\Lambda_k}\) are independent for all \(k \geq 2\) and disjoint \(\Lambda, \Lambda_1, \dots, \Lambda_k \Subset \R^d\).

Now, in \cite{Wu2000}, the author proves several functional inequalities for Poisson point processes and derives concentration bounds for Poisson functionals from them. Among these functional inequalities, we want to recapitulate the following {\em modified logarithmic Sobolev inequality}. 
Replacing the differential operator \(\nabla\) in the Gaussian world, the discrete gradients \(D_x\), acting as
\begin{align*}
    (D_x F)(\eta)
    = F(\eta + \delta_x) - F(\eta),
\end{align*} now come into play and the corresponding Dirichlet form is given by
\begin{align*}
    \mathcal{E}^{\pi}(f, g)
    = \nu\bigg[\int_{\R^d} \, (D_x f) (D_x g) \, \d x \bigg].
\end{align*} 
Note that the definition of \(\mathrm{Ent}_{\mu}[F]\) as
\begin{align*}
    \mathrm{Ent}_{\mu}[F]
    := \mu\big[\Phi(F)\big] - \Phi\big(\mu[F] \big),
\end{align*} still makes sense for any probability measure \(\mu \in \Pcal_\theta\) and measurable functions \(F \colon \Omega \to (0, \infty)\) of sufficient integrability.
As already discussed in the introduction of \cite{Wu2000} and first noted in \cite{Surgailis1984}, the naïve translation 
\begin{align*}
    \mathrm{Ent}_{\pi}[F^2]
    \lesssim \mathcal{E}^{\pi}(F, F)
\end{align*} of \Cref{equation:logarithmic_sobolev_inequality_Gross} to the Poisson point process of intensity one \(\pi\) fails due to the strictly-fatter-than-Gaussian tails of Poisson random variables.
However, the following inequality does hold, c.f.\ {\cite[Corollary 2.2]{Wu2000}},
\begin{align}\label{equation:MLSI_wu2000}
    \mathrm{Ent}_{\pi}[F]
    \leq \mathcal{E}^{\pi}(F, \log F)
\end{align} for every \(F \in L^1(\nu)\) with \(F > 0\) \(\pi\)-a.s..
In fact, the work \cite{Wu2000} even contains improvements over this inequality.
Nevertheless, the inequality \Cref{equation:MLSI_wu2000} occurs very naturally in the context of the {\em Ornstein--Uhlenbeck semigroup} \((T_t)_{t \geq 0}\) (see e.g.\ \cite{Last2014} for some results related to this semigroup) and that is why we will focus on it.
There, it governs the trend to equilibrium and exponential decay of relative entropy as follows.
Given the (formal) generator 
\begin{align*}
    (\mathscr{L} F)(\eta)
    = \int_{\R^d} \, \{F(\eta + \delta_x) - F(\eta) \} \, \d x
    + 
    \sum_{x \in \eta} \{F(\eta - \delta_x) - F(\eta) \},
\end{align*} of \((T_t)_{t \geq 0}\), describing dynamics where points are independently born into space and die with unit exponential rate, we have indeed
\begin{align*}
    \mathcal{E}^{\pi}(F, G)
    = \nu[F \, (-\mathscr{L} G) ]
\end{align*} by the Mecke equation. 
One quantity which in a way measures distance to the Poisson point process \(\pi\), but also is of independent interest in statistical mechanics, is the {\em specific relative entropy } \(\SpecEnt(\mu \,\vert\, \pi)\) of a measure \(\mu \in \Pcal_\theta\) w.r.t.\ \(\pi\). It is given by
\begin{align*}
    \SpecEnt(\mu \,\vert\, \pi)
    = \inf_{\Lambda \Subset \mathbb{R}^d} \tfrac{1}{\abs{\Lambda}}  I(\mu_{\Lambda} \,\vert\, \pi_{\Lambda})
    = \lim_{n \to \infty} \tfrac{1}{\abs{\Lambda_n}} \RelEnt_{\Lambda_n}(\mu \,\vert\, \pi).
\end{align*} where \(\Lambda_n = [-n, n]^d\) for \(n \in \mathbb{N}\).
Herein and always in this work, we denoted by \(\mu_{\Lambda}\), \(\Lambda \Subset \R^d\), the restriction of \(\mu\) to the sub-\(\sigma\)-algebra \(\Fcal_\Lambda\) generated by the local counting variables \(\{N_{\Delta} \,\colon\, \Delta \Subset \R^d,\, \Delta \subseteq \Lambda\}\), and \(I_{\Lambda}(\mu \,\vert\, \nu) = I(\mu_\Lambda \,\vert\, \nu_\Lambda)\) with
\begin{align*}
    \RelEnt(P \,\vert Q)
    = \begin{cases}
        Q\big[\Phi(f) \big], &\text{ if } f = \frac{\d P}{\d Q} \text{ exists },\\
        \infty, & \text{otherwise}.
    \end{cases}
\end{align*}
A formal calculation now shows that
\begin{align}\label{equation:derivative_of_relative_entropy}
    \frac{\d}{\d t} \SpecEnt(\mu T_t \,\vert\, \pi)
    = \mathcal{E}^{\pi}(T_t f, \log T_t f ),
\end{align} where \(\mu \in \Pcal_\theta\) is another probability measure and \(f = \frac{\d \mu}{\d \pi}\).
Hence, the modified logarithmic Sobolev inequality \Cref{equation:MLSI_wu2000} is (morally) equivalent to exponential decay of the (specific) relative entropy along trajectories of the Ornstein--Uhlenbeck semigroup.

A generalization of the above situation can be made to the world of {\em Gibbs measures}, a vast landscape of mathematical statistical mechanics and probability theory. Textbook references are e.g.\ \cite{Georgii2011, Dereudre2019}.
Recall that a (translation-invariant) {\em Gibbs point process}, or sometimes just {\em Gibbs measure}, is an element \(\nu\) of \(\Pcal_\theta\) which satisfies the {\em DLR equations} 
\begin{equation*}
    \int \, f(\omega) \, \nu(\d \omega)
    =
    \int \bigg[ \int  f(\eta) \, G_{\Lambda, \omega}(\d \eta) \bigg] \nu(\d \omega)
\end{equation*} for every measurable \(f \geq 0\) and every \(\Lambda \Subset \R^d\). Herein, \(G_{\Lambda, \omega}\) is the probability measure defined by
\begin{align*}
    \int \, f(\eta) \, G_{\Lambda, \omega}(\d \eta)
    = Z_{\Lambda, \omega}^{-1} \int \, f(\eta_\Lambda \omega_{\Lambda^c}) \, \e^{- H_\Lambda(\eta_\Lambda \omega_{\Lambda^c})} \, \pi(\d \eta)
\end{align*} where \(\pi\) is the Poisson point process of intensity one, \(Z_{\Lambda, \omega}^{-1}\) is a normalization constant and 
\begin{align*}
    H_\Lambda(\eta)
    := \lim_{n \to \infty} \big(H(\eta_{\Lambda_n})-H(\eta_{\Lambda_n\setminus\Lambda})\big)
\end{align*} is the conditional energy associated with a (translation-invariant) {\em energy function } \(H\). 
We will not consider issues of well-definedness and existence of Gibbs measures in the present work. Let us just mention two examples of energy functions for which these properties are well-known.
\begin{example}[Very nice pair potentials]\label[example]{example:very_nice_pair_potentials}
    Let \(\varphi \colon \R^d \to \R_+\) be a non-negative, compactly supported and even function. We can then consider the energy function
    \begin{align*}
        H(\omega) = \sum_{\substack{\{x,y\} \subseteq \omega, \\ x \neq y}} \varphi(x-y)
    \end{align*} for totally finite point configurations \(\omega\).
\end{example}
\begin{example}[Area interaction]\label[example]{example:area_interaction}
    Let \(R > 0\) and \(\gamma \in \R\setminus\{0\}\) be some fixed number.
    We can then consider the energy function
    \begin{align*}
        H(\omega) = \gamma \vert B_R(\omega)\vert
    \end{align*} for totally finite point configurations \(\omega\) and \(B_R(\omega) = \bigcup_{x \in \omega} B_R(x)\), \(B_R(x) = \{y \in \R^d \colon \abs{x-y} \leq R\}\). It is well-known, compare with e.g.~\cite{Ruelle1971, Giacomin1995, CCK1995}, that there are parameters \(\gamma > 0, R > 0\) such that more than one Gibbs measures exists for the corresponding area interaction \(H\) given above.
\end{example}
\begin{example}[Superstable pair interactions]\label[example]{example:superstable_pair_potentials}
    We use the same definitions as in \cite{Georgii1994}.
    Here, the energy function is again given as
    \begin{align*}
        H(\omega) = \sum_{\substack{\{x,y\} \subseteq \omega, \\ x \neq y}} \varphi(x-y)
    \end{align*} for totally finite point configurations \(\omega\), but with a much more general class of \(\varphi \colon \R^d \to \R\) than in \Cref{example:very_nice_pair_potentials}.
    We assume \(\varphi\) to be {\em superstable} and {\em regular}.
    This is e.g.\ fulfilled if \(\varphi\) is {\em non-integrably divergent at the origin} and {\em regular}, meaning respectively that there is some decreasing function \(\chi \colon (0, \infty) \to \mathbb{R}_+\) with \(\int_0^1 \, \chi(r) \, r^{d-1} \, \d r = \infty\) such that \(\varphi(x) \geq \chi(\abs{x})\) whenever \(\abs{x}\) is small enough and that there is some decreasing function \(\psi \colon \mathbb{R}_+ \to \mathbb{R}_+\) with \(\int_{0}^{\infty} \, \psi(r) \, r^{d-1} \, \d r < \infty\) such that
    \begin{enumerate}
        \item \(\varphi(x) \geq  -\psi(\abs{x})\) for all \(x \in \mathbb{R}^d\),
        \item \(\varphi(x) \leq \psi(\abs{x})\) for some \(r(\varphi) < \infty\) and all \(x \in \mathbb{R}^d\) with \(\abs{x} \geq r(\varphi)\).
    \end{enumerate}
    For well-definedness of the conditional energies and hence finite-volume Gibbs measures, we consider the class \(\Omega^\ast \subset \Omega\) of {\em tempered configurations}
    \begin{align*}
        \Omega^\ast 
        := \bigcup_{t > 0} \,\bigg\{\sum_{i \in \Lambda_n} N_{i + [-1,1]^d}^2 \,\leq\, t\abs{\Lambda_n} \,\, \text{ for all } n \in \mathbb{N} \bigg\}.
    \end{align*}
    We will call a probability measure \(\nu \in \Pcal_\theta\) a {\em tempered Gibbs measure} if it is a {\em tempered probability measure}, i.e.\ if \(\nu(\Omega^\ast) = 1\), and a Gibbs measure using the definition via DLR equations given above. 
\end{example}

Now, if we denote by \(h(x, \eta)\) the {\em conditional energy}
\begin{equation*}
    h(x,\eta)
    := \lim_{n\uparrow\infty}\big(H(\eta_{\Lambda_n}+\delta_x) - H(\eta_{\Lambda_n})\big)
\end{equation*} of a point \(x\in\R^d\) in a configuration \(\eta\in\Omega\), the {\em birth rate} \(b\) is given by
\begin{equation*}
    b(x,\eta) := \e^{-h(x,\eta)}.
\end{equation*} In point process language, what we call a birth rate here is a {\em Papangelou intensity}.
It gives rise to the following formal generator 
\begin{align*}
    (\mathscr{L} F)(\eta)
    = \int_{\R^d} \, b(x, \eta) \,\{F(\eta + \delta_x) - F(\eta) \} \, \d x
    + 
    \sum_{x \in \eta} \{F(\eta - \delta_x) - F(\eta) \},
\end{align*} which, at least formally, generates a semigroup admitting the corresponding Gibbs measures as reversible measures. This reversibility is now a consequence of the {\em GNZ equations} in place of the Mecke equation. 
Rigorous results regarding such continuum {\em birth-and-death dynamics} in that direction are contained e.g.\ in the articles \cite{JKSZ25, JKSZ26}. 

Now, we see as in \Cref{equation:derivative_of_relative_entropy} that a relevant quantity regarding the decay of relative entropy w.r.t.\ \(\nu\), along the trajectories of the dynamics associated with \(\mathscr{L}\), is given by the formal Dirichlet form expression
\begin{align}\label{equation:formal_dirichlet_form_birth_death}
    \mathcal{E}^{\nu}(F, \log F)
    = \nu\big[F (-\mathscr{L} \log F) \big]
    = \nu\bigg[\int_{\R^d} \, b(x, \cdot) \, ( D_x F ) (D_x \log F) \, \d x \bigg].
\end{align}

Note that in \Cref{example:very_nice_pair_potentials} and \Cref{example:area_interaction}, the birth rate \(b\) is bounded.
Hence, at least in these settings, the following inequality will be the relevant one.
We will say that \(\nu\) satisfies the modified logarithmic Sobolev inequality ({\bf MLSI--1}) with constant \(c_\nu > 0\), if
\begin{align}
    \mathrm{Ent}_{\nu}[F]
    \leq c_\nu \, \nu\bigg[\int_{\R^d} \, \, ( D_x F ) (D_x \log F) \, \d x \bigg]
\end{align} for all bounded local \(F > 0\).

Finite-volume versions of this inequality are, among other things, discussed in the article \cite{DaiPraPosta2013}. As an example, it is shown there that Gibbs measures for \Cref{example:very_nice_pair_potentials} obey such a bound for sufficiently low interaction strength (e.g.\ in  high-temperature regimes). In \cite{JKSZ26}, these techniques are applied to say the same for \Cref{example:area_interaction} and sufficiently high temperatures.

In the case that the birth rate \(b\) is unbounded, ({\bf MLSI--1}) does not follow from 
\begin{align}\label{equation:MLSI_correct_Dirichlet_form}
    \mathrm{Ent}_{\nu}[F]
    \leq c_\nu \, \mathcal{E}^{\nu}(F, \log F)
    = c_\nu \, \nu\bigg[\int_{\R^d} \, b(x, \cdot) \, (D_x F) \, (D_x \log F) \, \d x \bigg]
\end{align} for all bounded local \(F > 0\). In case that \Cref{equation:MLSI_correct_Dirichlet_form} holds for all bounded local \(F > 0\), we will say that \(\nu\) satisfies the the modified logarithmic Sobolev inequality ({\bf MLSI--b}) with constant \(c_\nu > 0\).

At high temperatures, there will be, as is well-known for these examples, just one Gibbs measure. On the other hand, for \Cref{example:area_interaction} also a phase transition, i.e.\ non-uniqueness of Gibbs measures at some parameters, is well-known.

Now, a natural question is:  
{\em Does the fact that some Gibbs measure \(\nu\) satisfies a ({\bf MLSI--1}) or a ({\bf MLSI--b})  with constant \(c_\nu > 0\) already imply that no other Gibbs measure w.r.t.\ the same chosen interaction (and parameters) exists?}

In \cite{CMRU2020, CR2022} the authors presented the following idea in the context of statistical mechanics of lattice systems, which we shall adapt to the present subject of investigation.
According to the {\em Gibbs variational principle}, the specific relative entropy of a Gibbs measure \(\mu\) w.r.t.\ another Gibbs measure \(\nu\) for the same specification is zero:
\begin{align*}
    \SpecEnt(\mu \,\vert\, \nu) 
    = \lim_{n \to \infty} \tfrac{1}{\abs{\Lambda_n}} \RelEnt(\mu_{\Lambda_n} \,\vert\, \nu_{\Lambda_n})
    = 0
\end{align*} and \(\SpecEnt(\widetilde{\mu} \,\vert\, \nu) = 0\) for another probability measure \(\widetilde{\mu}\) already implies that \(\widetilde{\mu}\) is in fact a Gibbs measure too.
Hence, if \(\nu\) is any Gibbs measure, an equivalent condition to its uniqueness is \(\SpecEnt(\mu \,\vert\, \nu)  > 0\) for every \(\mu \in \Pcal_\theta \setminus \{\nu\}\).
Employing the Donsker--Varadhan formula for the relative entropy, we see that
\begin{align*}
    \tfrac{1}{\abs{\Lambda_n}} \RelEnt(\mu_{\Lambda_n} \,\vert\, \nu_{\Lambda_n})
    &\,\geq\,
    \frac{1}{\abs{\Lambda_n}} \Big\{ \Big( \mu\big[F_n \big] - \nu\big[F_n \big] \Big) - \log \nu\big[ \e^{F_n - \nu[F_n]} \big] \Big\}
\end{align*} for every measurable \(F_n\) of our choice. Inspecting the right-hand side of this inequality, we can already see that it will be uniformly bounded from below if we are able to construct \((F_n)_{n \in \mathbb{N}}\) with roughly the following properties
\begin{enumerate}[topsep=0.3cm, itemsep=0.2cm]
    \item \( \tfrac{1}{\abs{\Lambda_n}} \mu[F_n] - \tfrac{1}{\abs{\Lambda_n}} \nu[F_n] \gg 0\),
    \item \(F_n\) concentrates sufficiently well, on an exponential scale, around its mean \(\nu[F_n]\) under \(\nu\).
\end{enumerate} The first property is easily achieved by choosing a bounded local observable \(f\) which separates the distributions \(\mu, \nu\) as in \(\mu[f] - \nu[f] > 0\) and then, exploiting stationarity, looking at its space-averages \(F_n = \int_{\Lambda_n} \, f \circ \theta_x \, \d x\). For the second property, we will have to use the corresponding concentration assumptions on \(\nu\) and take more care with the initial choice of \(f\) to apply them.

The contribution of this small work is the insight that ideas from \cite{CMRU2020, CR2022} carry over from the lattice situation to the continuum point process setting, also replacing the Gaussian concentration assumption with a Poisson-type one, and the technical work needed to enable this.

Indeed, we present the following
\begin{theorem}[MLSI--1 implies strictly positive specific relative entropy distance]\label[theorem]{theorem:MLSI_implies_distance_in_specific_relative_entropy}
    Let \(\nu \in \Pcal_\theta\) satisfy ({\bf MLSI--1}) with some constant \(c_\nu > 0\).
    Then, \(\SpecEnt(\mu \,\vert\, \nu) > 0\) for all \(\mu \in \Pcal_\theta \setminus \{\nu\}\).
\end{theorem}

This theorem now yields the following easy corollaries.
\begin{corollary}
    Let the Gibbs variational principle hold in the following sense: if \(\nu \in \Pcal_\theta\) is a Gibbs measure w.r.t.\ the fixed energy function \(H\) and \(\mu \in \Pcal_\theta\) is any probability measure with \(\SpecEnt(\mu \,\vert\, \nu) = 0\), then \(\mu\) is Gibbs measure w.r.t.\ \(H\) too.
    Then, if any Gibbs measure \(\nu\) w.r.t.\ \(H\) satisfies ({\bf MLSI--1}) with some constant \(c_\nu > 0\), there is no other Gibbs measure w.r.t.\ \(H\).
\end{corollary}

\begin{corollary}
    There are parameters \(\gamma \in \R\setminus\{0\}\) and \(R > 0\) such that none of the Gibbs measures \(\nu\) (which are more than one) w.r.t.\ the interaction in \Cref{example:area_interaction} can satisfy ({\bf MLSI--1}) with some constant \(c_\nu > 0\). 
\end{corollary}

In the case of unbounded birth rates, the following theorem becomes relevant.
\begin{theorem}[MLSI--b implies strictly positive specific relative entropy distance]\label[theorem]{theorem:MLSI_2_implies_distance_in_specific_relative_entropy}
    Let \(\nu\) be a (tempered, translation-invariant) Gibbs measure in the setting of \Cref{example:superstable_pair_potentials} and assume that \(\nu \in \Pcal_\theta\) satisfies ({\bf MLSI--b}) with some constant \(c_\nu > 0\).
    
    Then, \(\SpecEnt(\mu \,\vert\, \nu) > 0\) for all \(\mu \in \Pcal_\theta \setminus \{\nu\}\).
    In particular, there is no other (tempered, translation-invariant) Gibbs measure.
\end{theorem}
\begin{remark}
    The theorem is formulated for the case of superstable pair interactions as defined in \Cref{example:superstable_pair_potentials}, but the proof does in fact not directly rely on this specific form of the energy function. An additional input needed for this proof is just an upper bound of the form
    \begin{align*}
        \nu\big[(N_{\Lambda_n} / \abs{\Lambda_n}) \1_{N_{\Lambda_n} \,\geq\, t \abs{\Lambda_n}} \big]
        \,\lesssim\, \e^{- r(t) \abs{\Lambda_n}}
    \end{align*} where \(\sup_{t > 0} r(t) > 0\).
    It is here, in the setting of superstable pair interactions, provided by the upper bounds of an LDP for stationary empirical fields connected to the Gibbs variational principle, which is contained in \cite{Georgii1994}.
\end{remark}

Now turning to proofs, our \Cref{theorem:MLSI_implies_distance_in_specific_relative_entropy} is a direct consequence of the following two propositions.

Let us say that \(\nu\) satisfies ({\bf MGF--CI}) with constant \(c_\nu > 0\) if 
\begin{align*}
        \nu\big[\e^{\lambda (F - \nu[F])} \big]
        \leq 
        \exp\Big\{ c_{\nu} \alpha^2 \lambda \tfrac{\e^{\beta \lambda} - 1}{\beta}\Big\}.
\end{align*} for every \(F\) with \(D_x F \leq \beta\) \((\nu \otimes \d x)\)-a.s. and \(\int_{\R^d} \, \abs{D_x F}^2 \, \d x \leq \alpha^2\) \(\nu\)-a.s., for \(\alpha^2, \beta > 0\).

The first proposition is simply a restatement of \Cref{theorem:MLSI_implies_distance_in_specific_relative_entropy} with replaced assumptions:
\begin{proposition}[MGF bounds imply distance in specific relative entropy]\label[proposition]{proposition:MGF_bounds_imply_distance_in_specific_relative_entropy}
    Let \(\nu \in \Pcal_\theta\) satisfy ({\bf MGF--CI}) with some constant \(c_\nu > 0\).
    Then, \(\SpecEnt(\mu \,\vert\, \nu) > 0\) for all \(\mu \in \Pcal_\theta \setminus \{\nu\}\).
\end{proposition}
The second proposition shows that an MLSI for \(\nu\) implies MGF bounds under \(\nu\) for certain observables. Hence, \Cref{proposition:MGF_bounds_imply_distance_in_specific_relative_entropy} is applicable with the assumptions of \Cref{theorem:MLSI_implies_distance_in_specific_relative_entropy}. The derived bounds are of course equivalent to corresponding concentration-of-measure inequalities.
\begin{proposition}[MGF bounds from MLSI--1]\label[proposition]{proposition:centered_MGF_bound_from_MLSI}
    Assume that \(\nu \in \Pcal_\theta\) satisfies ({\bf MLSI--1}) with some constant \(c_\nu > 0\).
    Then, it also satisfies ({\bf MGF--CI}) with constant \(c_\nu > 0\).
\end{proposition}

The proofs of \Cref{proposition:MGF_bounds_imply_distance_in_specific_relative_entropy} and \Cref{proposition:centered_MGF_bound_from_MLSI} are given in the next \Cref{section:proofs}, specifically in \Cref{section:proofs_MLSI}. Subsequently and building on the preceding arguments, \Cref{theorem:MLSI_2_implies_distance_in_specific_relative_entropy} is proven in \Cref{section:proofs_MLSI_2}.

As a concluding remark, we want to say that this result further solidifies the need for different and new techniques to understand {\em quantitatively} trend to equilibrium of stochastic dynamics such as the considered birth-and-death dynamics and especially decay of specific relative entropy in regimes of non-uniqueness of (the reversible) Gibbs measures. Some of the {\em qualitative} basic building blocks in regimes of non-uniqueness were already laid in \cite{JKSZ25, JKSZ26}.

\section{Proofs}\label[section]{section:proofs}
The proofs of \Cref{proposition:MGF_bounds_imply_distance_in_specific_relative_entropy}, \Cref{proposition:centered_MGF_bound_from_MLSI} and hence \Cref{{theorem:MLSI_implies_distance_in_specific_relative_entropy}} are contained in \Cref{section:proofs_MLSI}. Enhancing the approach there with some new ideas to deal with the unbounded birth rate \(b\), \Cref{section:proofs_MLSI_2} contains the proof of \Cref{theorem:MLSI_2_implies_distance_in_specific_relative_entropy}.

\subsection{Dirichlet form with birth rate \(1\)}\label[section]{section:proofs_MLSI}
In our first proof, we detail how we adapt ideas of \cite{CMRU2020, CR2022} to our setting. Heuristics regarding this were already given in \Cref{section:introduction_and_main_theorems}.
\begin{proof}[Proof of \Cref{proposition:MGF_bounds_imply_distance_in_specific_relative_entropy}]
    Consider \(\mu \in \Pcal_\theta \setminus \{\nu\}\).
    Then, we can separate \(\mu\) and \(\nu\) as follows.
    There is an \(r > 0\) and \(g \geq 0\), \(g \in C_c(\R^d)\) with \(\mathrm{supp}\, g \subseteq \Lambda_r\) such that
    \begin{align*}
        \varrho 
        := \mu\big[f \big] - \nu\big[ f \big]
        > 0
    \end{align*}
    for
    \begin{align*}
        f(\eta)
        = \pm \e^{-\sum_{x \in \eta} g(x)},
    \end{align*} where the sign is chosen to make \(\varrho\) positive.

    Now set
    \begin{align}\label{equation:test_function_F_n}
        F_n := \int_{\Lambda_n} \, (f \circ \theta_x) \, \d x.
    \end{align} 
    Then,
    \begin{align*}
        (D_z f)(\eta)
        = (e^{-g(z)} - 1) \, f(\eta)
    \end{align*} and hence
    \begin{align*}
        \abs{(D_z F_n)(\eta)}
        &= \bigg\vert \int_{\Lambda_n} \, (D_{z-x} f)(\theta_x \eta) \, \d x \bigg\vert
        \leq \int_{\Lambda_n} \,  \abs{1 - \e^{-g(z-x)}} \abs{f(\theta_x \eta)} \, \d x \\
        &\leq \int_{\R^d} \,  \abs{1 - \e^{-g(x)}} \, \d x
        =: \beta.
    \end{align*} With \(\psi(x) := \sup_{\omega} \vert (D_{x} f)(\omega) \vert\), we also get the following bound which we prove in \Cref{lemma:L2_norm_gradient_of_averaged_f} below:
    \begin{align*}
        \int_{\R^d} \abs{D_z F_n}^2 \d z
        \leq \norm{ \psi }_{L^1(\R^d)}^2 \abs{\Lambda_n} 
        \leq \beta^2 \abs{\Lambda_n} =: \alpha_n^2
    \end{align*}

    Then, by the Donsker--Varadhan formula for the relative entropy in combination with \(\Fcal_{\Lambda_{n+r}}\)-measurability of \(F\), and by translation-invariance
    \begin{align*}
        \frac{1}{\abs{\Lambda_n}} I_{\Lambda_{n+r}}( \mu \,\vert\, \nu )
        &\geq \frac{1}{\abs{\Lambda_n}} \Big\{ \Big( \mu\big[F_n \big] - \nu\big[F_n \big] \Big) - \log \nu\big[ \e^{F_n - \nu[F_n]} \big] \Big\} \\
        &= \Big( \mu\big[f \big] - \nu\big[f \big] \Big) 
        - \frac{1}{\abs{\Lambda_n}} \log \nu\big[ \e^{F_n - \nu[F_n]} \big].
    \end{align*}  
    Now replacing \(f\) by \(\lambda f\) for \(\lambda > 0\) yields
    \begin{align*}
        \tfrac{1}{\abs{\Lambda_n}} I_{\Lambda_{n+r}}( \mu \,\vert\, \nu ) 
        \,\geq\, \sup_{\lambda > 0} \Big\{
        \varrho \lambda
        - \tfrac{1}{\abs{\Lambda_n}} \log \nu\big[ \e^{\lambda (F_n - \nu[F_n])} \big]
        \Big\}.
    \end{align*}
    From the assumption ({\bf MGF--CI}) on \(\nu\), we get
    \begin{align*}
        \nu\big[ \e^{\lambda (F_n - \nu[F_n])} \big] 
        \leq \exp\Big\{ c_{\nu} \alpha_n^2 \lambda \tfrac{\e^{\beta \lambda} - 1}{\beta}\Big\}
        = \exp\Big\{ c_{\nu} \beta^2 \abs{\Lambda_n} \lambda \tfrac{\e^{\beta \lambda} - 1}{\beta}\Big\}.
    \end{align*}
    Finally,
    \begin{align*}
        \frac{1}{\abs{\Lambda_n}} I_{\Lambda_{n+r}}( \mu \,\vert\, \nu ) 
        \,\geq \, 
        \sup_{\lambda > 0} \Big\{
        \varrho \lambda
        - (c_{\nu} \beta^2 \tfrac{\e^{\beta \lambda} - 1}{\beta \lambda} ) \lambda^2
        \Big\} 
        > 0,
    \end{align*} independently of \(n \in \mathbb{N}\), giving
    \begin{align*}
        \SpecEnt( \mu \,\vert\, \nu ) 
        = \inf_{n \in \mathbb{N}} \tfrac{1}{\abs{\Lambda_{n+r}}} I_{\Lambda_{n+r}}( \mu \,\vert\, \nu ) 
        = \inf_{n \in \mathbb{N}} \tfrac{1}{\abs{\Lambda_{n}}} I_{\Lambda_{n+r}}( \mu \,\vert\, \nu ) 
        > 0.
    \end{align*}
\end{proof}

The following technical lemma was needed in the above proof. The statement and proof are very close to \cite[Lemma A.1]{CMRU2020} in the lattice setting, but we still incorporate it for convenience of the reader.
\begin{lemma}\label[lemma]{lemma:L2_norm_gradient_of_averaged_f}
    Let \(f\) be a bounded local function and 
    \begin{align*}
        F
        := \int_{\Lambda} \, (f \circ \theta_x) \, \d x.
    \end{align*}
    Then
     \begin{align*}
        \int_{\R^d} \abs{D_z F}^2 \d z
        \leq \norm{ \psi }_{L^1(\R^d)}^2 \abs{\Lambda} 
    \end{align*} with \(\psi(x) := \sup_{\omega} \vert (D_{x} f)(\omega) \vert\).
\end{lemma}
\begin{proof}
    We have
    \begin{align*}
        \abs{(D_z F)(\eta)}
        &= \bigg\vert \int_{\Lambda} \, f(\theta_x \eta + \delta_{z-x}) \, \d x - \int_{\Lambda} \, f(\theta_x \eta) \, \d x \bigg\vert \\
        &\leq \int \, \psi(z-x)  \, \1_{\Lambda}(x) \, \d x
        = (\psi \ast \1_{\Lambda})(z)
    \end{align*} 
    Young's convolution inequality now implies
    \begin{align*}
        \int_{\R^d} \abs{D_z F}^2 \d z 
        &\leq \norm{\psi \ast \1_{\Lambda}}_{L^2(\R^d)}^2 
        \leq \norm{ \psi }_{L^1(\R^d)}^2  \norm{ \1_{\Lambda} }_{L^2(\R^d)}^2 \\
        &= \norm{ \psi }_{L^1(\R^d)}^2 \abs{\Lambda} .
    \end{align*}
\end{proof}

Finally, we show how to convert MLSI into corresponding MGF / concentration bounds.
\begin{proof}[Proof of \Cref{proposition:centered_MGF_bound_from_MLSI}]
    We follow the so-called Herbst's argument, as seen in the proof of e.g.\ \cite[Proposition 3.1]{Wu2000}.
    Set
    \begin{align*}
        H(\lambda)
        := \frac{1}{\lambda} \log \nu\big[\e^{\lambda F} \big]
    \end{align*}
    and notice that \(H(0^+) = \nu[F]\) as well as
    \begin{align*}
        H'(\lambda)
        = \frac{\mathrm{Ent}_{\nu}[e^{\lambda F}]}{\lambda^2 \nu[\e^{\lambda F}]}.
    \end{align*}
    Now, 
    \begin{align*}
        \mathrm{Ent}_{\nu}[e^{\lambda F}]
        &\leq c_\nu \, \nu\bigg[\int_{\R^d} \, \, (D_x \e^{\lambda F}) \, (D_x \lambda F) \, \d x \bigg]
        \leq c_\nu \, \nu\bigg[\int_{\R^d} \, \, (D_x \lambda F)^2 \e^{\lambda F} \e^{\lambda \abs{D_x F}}  \, \d x \bigg] \\
        &\leq c_\nu \alpha^2 \lambda^2 \e^{\beta \lambda} \nu\big[\e^{\lambda F} \big],
    \end{align*} where we used the mean-value theorem for the second inequality.
    It follows that
    \begin{align*}
        H'(\lambda)
        \leq  c_\nu \alpha^2 \e^{\beta \lambda}
    \end{align*} and hence
    \begin{align}\label{equation:MGF_bound_Herbst_argument}
        \nonumber\nu\big[\e^{\lambda F} \big]
        &= \exp\big\{ \lambda H(\lambda) \big\} 
        = \exp\Big\{ \lambda H(0) + \lambda \int_{0}^{\lambda} H'(\widetilde{\lambda}) \d\widetilde{\lambda} \Big\} \\ 
        &\leq \exp\Big\{ \lambda \,\nu[F] + c_{\nu} \alpha^2 \lambda \tfrac{\e^{\beta \lambda} - 1}{\beta} \Big\}.
    \end{align}
\end{proof}

\subsection{Dirichlet form with birth rate \(b\)}\label[section]{section:proofs_MLSI_2}
We will finally analyze the case of ({\bf MLSI--b}), where the generally unbounded birth rate / Papangelou intensity \(b\) appears in the corresponding Dirichlet form. It is recommended to first read the proofs in \Cref{section:proofs_MLSI} to facilitate understanding.

\begin{proof}[Proof of \Cref{theorem:MLSI_implies_distance_in_specific_relative_entropy}]
    We follow the reasoning of the proof of \Cref{proposition:MGF_bounds_imply_distance_in_specific_relative_entropy} and adapt the technique of the proof of \Cref{proposition:centered_MGF_bound_from_MLSI} to the specific \(F_n\) as in \Cref{equation:test_function_F_n}.
    Indeed, we see as in \Cref{equation:MGF_bound_Herbst_argument} that 
    \begin{align*}
         \nu\big[\e^{\lambda (F_n - \nu[F_n])} \big]
         \leq \exp\bigg\{ c_{\nu} \,\lambda \e^{\beta \lambda} \int_{0}^{\lambda} \, K_{n}(\widetilde{\lambda}) \,\d\widetilde{\lambda} \bigg\}
    \end{align*} with
    \begin{align*}
        K_{n}(\lambda)
        := \frac{1}{\nu\big[\e^{\lambda F_n} \big]} \,\nu\bigg[\e^{\lambda F_n} \bigg(\int_{\R^d} \, b(z, \cdot) \abs{D_z F_n}^2 \, \d z  \bigg) \bigg].
    \end{align*}
    Our task is now to give bounds for \(K_n(\lambda) / \abs{\Lambda_n}\) which are somewhat uniform in \(n\) and valid for small \(\lambda\). 
    
    Define \(\varphi(x) = 1 - \e^{-g(x)}\) and observe that
    \begin{align*}
        &\nu\bigg[\e^{\lambda F_n} \bigg(\int_{\R^d} \, b(z, \cdot) \abs{D_z F_n}^2 \, \d z  \bigg) \bigg] 
        \leq \nu\bigg[\e^{\lambda F_n} \bigg(\int_{\R^d} \, b(z, \cdot) \abs{\varphi \ast \1_{\Lambda_n}(z)}^2 \, \d z  \bigg) \bigg] \\
        &= \nu\bigg[\sum_{z \in \eta} \e^{\lambda F_n(\eta - \delta_z)}  \abs{\varphi \ast \1_{\Lambda_n} (z)}^2 \bigg] 
        = \nu\bigg[\e^{\lambda F_n(\eta)} \sum_{z \in \eta}  \e^{\lambda (D_z F_n)(\eta - \delta_z)} \abs{\varphi \ast \1_{\Lambda_n} (z)}^2 \bigg]  \\
        &\leq \e^{\beta \lambda } \, \nu\bigg[\e^{\lambda F_n(\eta)} \sum_{z \in \eta} \abs{\varphi \ast \1_{\Lambda_n} (z)}^2 \bigg] 
        \leq \beta \e^{\beta \lambda }  \, \nu\bigg[\e^{\lambda F_n} \, N_{\Lambda_n + \mathrm{supp}\, \varphi} \bigg], 
    \end{align*} where the first identity is by the GNZ equations.

    Now define the {\em stationary empirical fields } \(R_{n, \omega} \in \Pcal_\theta\) by
    \begin{align*}
        R_{n, \omega}
        := \frac{1}{\abs{\Lambda_n}} \int_{\Lambda_n} \, \delta_{\theta_x \omega^{(n)}} \, \d x,
    \end{align*} where 
    \begin{align*}
        \omega^{(n)} 
        := \sum_{\substack{x \in \omega_{\Lambda_n}, \\ i \in \Z^d}} \delta_{x + 2n i}
    \end{align*} is the \(\Lambda_n\)-periodization of \(\omega_{\Lambda_n}\).

    Now let \(k\) be the smallest \(j \in \mathbb{N}\) such that \(\Lambda_{n} + \mathrm{supp} \, g \subseteq \Lambda_{n+j}\) and let us use the following crude bound valid for any \(T > 0\) and \(\lambda \in (0,1)\):
    \begin{align*}
        &\frac{1}{\nu\big[\e^{\lambda F_n} \big]} \,\nu\bigg[\e^{\lambda F_n} \, N_{\Lambda_n + \mathrm{supp}\, \varphi} \bigg] 
        \leq \frac{1}{\nu\big[\e^{\lambda F_n} \big]} \,\nu\bigg[\e^{\lambda F_n} \, N_{\Lambda_n + k} \bigg]  \\
        &= \abs{\Lambda_{n+k}} \frac{1}{\nu\big[\e^{\lambda F_n} \big]} \,\nu\bigg[\e^{\lambda F_n} \, R_{n+k, \cdot}[N_C] \bigg] \\
        &\leq \abs{\Lambda_{n+k}} \bigg\{ T + \e^{2 \lambda \norm{f}_\infty \abs{\Lambda_{n+k}} } \, \nu\Big[\abs{\Lambda_{n+k}} R_{n+k, \cdot}[N_{C}] \1_{R_{n+k, \cdot}[N_C] \geq T} \Big] \bigg\} \\
        &\leq  \abs{\Lambda_{n+k}} \bigg\{ T + \e^{- (T - 2 \norm{f}_\infty) \abs{\Lambda_{n+k}} } \, \nu\Big[\e^{2 \abs{\Lambda_{n+k}} R_{n+k, \cdot}[N_C]}\Big] \bigg\} 
        \end{align*} with \(C = [0,1]^d\).
        
    We will now apply the upper bound of the LDP \cite[Theorem 3 (b)]{Georgii1994} for \(R_{n,\cdot}\) under \(\nu\) with rate function \(I_{1,1}\), for all \(P \in \Pcal_\theta\) given by
    \begin{align*}
        I_{1,1}(P)
        = \SpecEnt(P \,\vert\, \pi) +  H(P) - \min_{Q \in \Pcal_\theta} [\SpecEnt(Q \,\vert\, \pi) +  H(Q)] 
        \geq 0
    \end{align*} where
    \begin{align*}
        H(P)
        := \lim_{n \to \infty} \tfrac{1}{\abs{\Lambda_n}} P[H_{\Lambda_n}].
    \end{align*}
    We get that 
    \begin{align}\label{equation:tilted_LDP}
         \limsup_{n \to \infty} \tfrac{1}{\abs{\Lambda_{n+k}}} \log \nu\Big[\e^{2 \abs{\Lambda_{n+k}} R_{n+k, \cdot}[N_C]}\Big] 
         \leq - \inf_{P \in \Pcal_\theta} J(P)
    \end{align} with
    \begin{align*}
        J(P)
        &:= I_{1, 1}(P) - 2 P[N_C] \\
        &= \SpecEnt(P \,\vert\, \pi) +  H(P) - 2 P[N_C] - c,
    \end{align*} where \(c \in \mathbb{R}\) is a finite constant.
    We now only have to argue that the infimum on the right-hand side of \Cref{equation:tilted_LDP} is not \(-\infty\).
    To see that: stability of \(H\) yields
    \begin{align*}
        H(P)
        \geq - b P[N_C]
    \end{align*} for some fixed \(b > 0\).
    But, for any \(a > 0\),
    \begin{align*}
        a P[N_C]
        = \tfrac{1}{\abs{\Lambda_n}}  P[ a N_{\Lambda_n}]
        \,\leq\, \tfrac{1}{\abs{\Lambda_n}} I(P_{\Lambda_n} \,\vert\, \pi_{\Lambda_n}) + \tfrac{1}{\abs{\Lambda_n}} \log \pi[\e^{a N_{\Lambda_n}}]
    \end{align*} and hence
    \begin{align*}
        \SpecEnt(P \,\vert\, \pi)
        \geq a P[N_C] - (\e^{a} - 1).
    \end{align*}
     It follows, with e.g.\ \(a = b+2\), that
    \begin{align*}
        J(P)
        \,&\geq\, (a-b-2) P[N_C] - (\e^{a} - 1) - c \\
        &\geq - (\e^{b+2} - 1) - c.
    \end{align*} for all \(P \in \Pcal_\theta\).
    
    If we now choose \(T = 2 (\norm{f}_\infty + \e^{b+2} - 1 + c)\), the statement of the theorem follows again along the same lines as in the proof of \Cref{proposition:MGF_bounds_imply_distance_in_specific_relative_entropy} because 
    \begin{align*}
        K_n(\lambda)
        \leq  (T+1) \beta \e^{\beta \lambda} \abs{\Lambda_{n+k}}
    \end{align*} for \(n\) large enough, uniformly in \(\lambda \in (0,1)\).
\end{proof}

\bibliographystyle{alpha}
\bibliography{references}

\end{document}